\documentclass{amsart}

\usepackage[colorlinks=true,linkcolor=magenta,citecolor=blue]{hyperref}
 
\usepackage{amsfonts}
\usepackage{amssymb}
\usepackage{amsmath}
\usepackage{amsthm}

\newtheorem{theorem}{Theorem}[section]
\newtheorem{proposition}[theorem]{Proposition}
\newtheorem{lemma}[theorem]{Lemma}
\newtheorem{corollary}[theorem]{Corollary}

\theoremstyle{definition}

\theoremstyle{remark}
\newtheorem{rmk}[theorem]{Remark}

\newenvironment{lem}{\begin{lemma}}{\end{lemma}}
\newenvironment{prop}{\begin{proposition}}{\end{proposition}}

\newcommand{\sfX}{\mathsf{X}}

\usepackage{tikz-cd}

\usepackage{etoolbox}
\AtBeginEnvironment{bmatrix}{\setlength{\arraycolsep}{2pt}}

\newcommand{\pl}{{\!+\!}}
\newcommand{\mn}{{\!-\!}}

\begin{document}

\title{The Ext-algebra of the Brauer tree algebra associated to a line}

\author{Olivier Dudas}
\address{Universit\'e de Paris and Sorbonne Universit\'e, CNRS, IMJ-PRG, F-75006 Paris, France.}
\email{olivier.dudas@imj-prg.fr}

\thanks{The author gratefully acknowledges financial support by the ANR, Project
No ANR-16-CE40-0010-01.}

\maketitle

\begin{abstract}
We compute the $\mathsf{Ext}$-algebra of the Brauer tree algebra associated to a line with no exceptional vertex. 
\end{abstract}

\section*{Introduction}
This note provides a detailed computation of the $\mathsf{Ext}$-algebra for a very specific finite dimensional algebra, namely a
Brauer tree algebra associated to a line, with no exceptional vertex. Such algebras appear for example as the principal $p$-block of the symmetric group $\mathfrak{S}_p$, and in a different context, as blocks of the Verlinde categories $\mathsf{Ver}_{p^2}$ studied by Benson--Etingof in \cite{BE} (our computation is actually motivated by \cite[Conj. 1.3]{BE}).

\smallskip

Let us emphasise that $\mathsf{Ext}$-algebras for more general biserial algebras were explicitly computed by Green--Schroll--Snashall--Taillefer in \cite{GSST}, but under some assumption on the multiplicity of the the vertices, assumption which is not satisfied  for the simple example treated in this note. Other general results relying on Auslander--Reiten theory were obtained by Antipov--Generalov \cite{AG} and Brown \cite{Br}. However we did not manage to use their work to get an explicit description in our case. Nevertheless, the simple structure of the projective indecomposable modules for the line allows a straightforward approach using explicit projective resolutions of simple modules. The Poincar\'e series for the $\mathsf{Ext}$-algebra is given in Proposition~\ref{prop:poincare} and its structure as a path algebra with relations is given in Proposition~\ref{prop:main}.

\section*{Acknowledgments}
We thank Rapha\"el Rouquier and Rachel Taillefer for providing helpful references.

\section{Notation}
Let $\mathbb{F}$ be a field, and $A$ be the self-injective finite dimensional $\mathbb{F}$-algebra. All $A$-modules will be assumed to be finitely generated.
Given an $A$-module $M$, we denote by $\Omega(M)$ the kernel of a projective cover $P\twoheadrightarrow M$. Up to isomorphism it does not depend on the cover. We then define inductively $\Omega^n(M) = \Omega(\Omega^{n-1}(M))$ for $n \geq 1$.

\smallskip
To compute the extension groups between simple modules we will use the property that if $\Omega^n(M)$ is indecomposable and non-projective then
$$\mathsf{Ext}_A^n(M,S) \simeq \mathsf{Hom}_{A}(\Omega^n(M),S)$$ 
for all simple $A$-module $S$ and all $n \geq 1$.

\smallskip

For computing the algebra structure on the various $\mathsf{Ext}$-groups it will be convenient to work in the homotopy category $\mathsf{Ho}(A)$ of the complexes of finitely generated $A$-modules. If $S$ (resp. $S'$) is a simple $A$-module, and ${P}_\bullet \rightarrow S$ (resp. $P_\bullet' \rightarrow S'$) is a projective resolution then 
$$ \mathsf{Ext}_A^n(S,S') \simeq \mathsf{Hom}_{\mathsf{Ho(A)}}(P_\bullet,P_\bullet'[n])$$
with the Yoneda product being given by the composition of maps in $\mathsf{Ho}(A)$.

\smallskip
Assume now that $A$ is the $\mathbb{F}$-algebra associated to the following Brauer tree with $N+1$ vertices:
$$
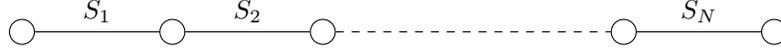
\begin{tikzpicture}
    \node[shape=circle,draw=black] (A) at (0,0) {};
    \node[shape=circle,draw=black] (B) at (2,0) {};
    \node[shape=circle,draw=black] (C) at (4,0) {};
    \node[shape=circle,draw=black] (D) at (8,0) {};
    \node[shape=circle,draw=black] (E) at (10,0) {};

    \path (A) edge node[above] {$S_1$} (B);
    \path (B) edge node[above] {$S_2$} (C);
    \path[dashed] (C) edge node[above] {} (D);
    \path (D) edge node[above] {$S_N$} (E);
\end{tikzpicture}$$
Here, unlike in \cite{GSST} we assume that there are no exceptional vertex. The edges are labelled by the simple $A$-modules $S_1, \ldots,S_{N}$. 
The head and socle of $P_i$ are isomorphic to $S_i$ and $\mathsf{rad}(P_i)/S_i \simeq S_{i-1}\oplus S_{i+1}$ with the convention that $S_0 = S_{N+1} = 0$. 

\section{Ext-groups}

Given $1 \leq i \leq j \leq N$ with $i-j $ even, there is, up to isomorphism, a unique non-projective indecomposable module ${}^i \sfX^j$ such that 
\begin{itemize}
  \item $\mathsf{rad}({}^i \sfX^j) = S_{i+1} \oplus S_{i+3} \oplus \cdots \oplus S_{j-1}$
  \item $\mathsf{hd}({}^i \sfX^j) = S_{i} \oplus S_{i+2} \oplus \cdots \oplus S_{j}$.
\end{itemize}
The structure of ${}^i \sfX^j$ can be represented by the following diagram:
$$\begin{tikzcd}
&[-20pt] S_i \ar[rdd,dash]&[-30pt] &[-30pt] S_{i+2} \ar[rdd,dash]&[-30pt] &[-30pt] S_{i+4}&[-20pt] \cdots &[-20pt] S_{j-2} \ar[rdd,dash] &[-30pt] &[-30pt] S_j  \\[-15pt] 
{}^i \sfX^j = & & & & & &  \cdots & &  \\[-15pt]
& & S_{i+1} \ar[ruu,dash] & & S_{i+3} \ar[ruu,dash] & & \cdots & & S_{j-1} \ar[ruu,dash]
\end{tikzcd}$$
Similarly we denote by ${}_i \sfX_j$ the unique indecomposable module with the following structure:
$$\begin{tikzcd}
&[-20pt] &[-30pt] S_{i+1} \ar[rdd,dash] &[-30pt] &[-30pt] S_{i+3} \ar[rdd,dash] &[-30pt] &[-20pt] \cdots &[-20pt] &[-30pt] S_{j-1} \ar[rdd,dash] &[-30pt] \\[-15pt] 
{}_i \sfX_j = & & & & & &  \cdots & &  \\[-15pt]
&S_i \ar[ruu,dash]&  & S_{i+2} \ar[ruu,dash]&  &  S_{i+4}&  \cdots &  S_{j-2} \ar[ruu,dash] &  & S_j 
\end{tikzcd}$$
Finally, in the case where $i-j $ is odd we define the modules ${}_i \sfX^j$ and ${}^i \sfX_j$ as the indecomposable modules with the following respective structure:
$$\begin{tikzcd}
&[-20pt] &[-30pt] S_{i+1} \ar[rdd,dash] &[-30pt] &[-30pt] S_{i+3} \ar[rdd,dash] &[-30pt] &[-20pt] \cdots &[-20pt] &[-30pt] S_{j}   \\[-15pt] 
{}_i \sfX^j = & & & & & &  \cdots & &  \\[-15pt]
&S_i \ar[ruu,dash]&  & S_{i+2} \ar[ruu,dash]&  &  S_{i+4}&  \cdots &  S_{j-1} \ar[ruu,dash] &  
\end{tikzcd}$$
$$\begin{tikzcd}
&[-20pt] S_i \ar[rdd,dash]&[-30pt] &[-30pt] S_{i+2} \ar[rdd,dash]&[-30pt] &[-30pt] S_{i+4}&[-20pt] \cdots &[-20pt] S_{j-1} \ar[rdd,dash] &[-30pt]   \\[-15pt] 
{}^i \sfX_j = & & & & & &  \cdots & &  \\[-15pt]
& & S_{i+1} \ar[ruu,dash] & & S_{i+3} \ar[ruu,dash] & & \cdots & & S_{j} 
\end{tikzcd}$$
For convenience we will extend the notation ${}^i \sfX^j$, ${}_i \sfX_j$, ${}_i \sfX^j$ and ${}^i \sfX_j$ to any integers $i,j \in \mathbb{Z}$ (with the suitable parity condition on $i-j$) so that the following relations hold:
\begin{equation} \label{eq:xij}
{}^i \sfX =  {}_{1-i} \sfX, \qquad {}^i \sfX^j = {}_{j} \sfX_i, \qquad
 {}^{i \pm 2N} \sfX = {}^i \sfX.
\end{equation}
Note that this also implies $\sfX^{j} = \sfX_{1-j}$ and $\sfX^{j\pm 2N} = \sfX^j$.

\begin{lem}\label{lem:omegaxij}
Let $i,j \in \mathbb{Z}$ with $i-j$ even. Then 
$$\Omega ( {}^i \sfX^j) \simeq  {}^{i-1} \sfX^{j+1}.$$
\end{lem}

\begin{proof}
Using the relations \eqref{eq:xij} it is enough to prove that for $1 \leq k \leq l \leq N$ we have the following isomorphisms
$$ \Omega ( {}^k \sfX^l)  \simeq  {}^{k-1} \sfX^{l+1}, \quad 
\Omega ( {}_k \sfX^l) \simeq  {}_{k+1} \sfX^{l+1}, \quad 
 \Omega ( {}^k \sfX_l) \simeq  {}^{k-1} \sfX_{l-1}, \quad
 \Omega ( {}_k \sfX_l) \simeq  {}_{k+1} \sfX_{l-1}.$$
We only consider the first one, the others are similar. If $1 \leq k \leq l \leq N$ a projective cover of  ${}^k \sfX^l$ is given by $P_k \oplus P_{k+2} \oplus \cdots \oplus P_l \twoheadrightarrow {}^k \sfX^l$, whose kernel equals ${}^{k-1} \sfX^{l+1}$. Note that this holds even when $k=1$ since ${}^0\sfX^{l+1} = {}_1 \sfX^{l+1}$ or when $l = N$ since ${}^{k-1} \sfX^{N+1} = {}^{k-1} \sfX^{-N+1} = {}^{k-1} \sfX_{N} $.
\end{proof}

We deduce from Lemma~\ref{lem:omegaxij} that for any simple module $S_i$ and for all $k \geq 0$ we have
$$ \Omega^k (S_i) = \Omega^k ({}^i\sfX^i) \simeq {}^{i-k} \sfX^{i+k}$$ 
as $A$-modules. Consequently we have 
$$\mathsf{Ext}_A^k(S_i,S_j) = \left\{ \begin{array}{ll} \mathbb{F} & \text{if $S_j$ appears in the head of ${}^{i-k} \sfX^{i+k}$}, \\ 0 & \text{otherwise.} \end{array}\right.$$
From this description one can compute explicitly the Poincar\'e series of the $\mathsf{Ext}$-groups.

\begin{prop}\label{prop:poincare}
Given $1 \leq i,j\leq N$, the Poincar\'e series of $\mathsf{Ext}_A^\bullet(S_i,S_j)$ is given by
$$ \sum_{k \geq 0} \mathsf{dim}_\mathbb{F} \, \mathsf{Ext}_A^k(S_i,S_j) t^k= \frac{Q_{i,j}(t) + t^{2N-1} Q_{i,j}(t^{-1})}{1-t^{2N}}$$
where $Q_{i,j}(t) =  t^{|j-i|} + t^{|j-i|+2} + \cdots + t^{N-1-|N+1-j-i|}$.
\end{prop}

\begin{proof} Without loss of generality we can assume that $i \leq j$. 
Let $k \in \{0,\ldots,N-1\}$. If $i+j \leq N+1$, the simple module $S_j$ appears in the head of  ${}^{i-k} \sfX^{i+k}$ if and only if $k=j-i, j-i+2, \ldots, j+i-2$. The limit cases are indeed ${}^{2i-j} \sfX^{j}$ for $k = j-i$ and ${}^{2-j} \sfX^{2i+j-2} = {}_{j-1}\sfX^{2i+j-2}$ for $k = j+i-2$. Note that if $j-i \leq k \leq i+j-2$ then $j \leq i+k$ and $j \leq 2N-i-k$ so that $S_j$ appears in the head of ${}^{i-k} \sfX^{i+k} = {}^{i-k} \sfX_{2N-i-k+1}$ whenever $k$ has the suitable parity.  If $i+j > N+1$ one must ensure that $j \leq 2N-i-k$ and therefore $S_j$ appears in the head of  ${}^{i-k} \sfX^{i+k}$ if and only if $k=j-i, j-i+2, \ldots, 2N-i-j$. Consequently we have
\begin{equation} \label{eq:ext}
\begin{aligned}
\sum_{k = 0}^{N-1} \mathsf{dim}_\mathbb{F} \, \mathsf{Ext}_A^k(S_i,S_j) t^k & \, = t^{j-i} + t^{j-i+2} + \cdots + t^{N-1-|N+1-j-i|} \\[-10pt]
& \, =  t^{|j-i|} + t^{|j-i|+2} + \cdots + t^{N-1-|N+1-j-i|} \\
& \, = Q_{i,j}(t).\end{aligned}
\end{equation}
Now the relation 
$$\Omega^N (S_i)  = {}^{i-N} \sfX^{i+N} = {}_{1+N-i}\sfX_{1-N-i} = {}_{1+N-i}\sfX_{1+N-i} = S_{N+1-i}$$ 
yields
$$\sum_{k = 0}^{2N-1} \mathsf{dim}_\mathbb{F} \, \mathsf{Ext}_A^k(S_i,S_j) t^k = \sum_{k = 0}^{N-1} \mathsf{dim}_\mathbb{F} \, \mathsf{Ext}_A^k(S_i,S_j) t^k + t^N \sum_{k = 0}^{N-1} \mathsf{dim}_\mathbb{F} \, \mathsf{Ext}_A^k(S_{N+1-i},S_j) t^k.$$
and the proposition follows from \eqref{eq:ext} after observing that $Q_{N+1-i,j}(t) = t^{N-1} Q_{i,j}(t^{-1})$.
\end{proof}

\section{Algebra structure}

\subsection{Minimal resolution}
Given $1\leq i \leq N- 1$ we fix non-zero maps $f_i : P_{i} \longrightarrow P_{i+1}$ and
$f_i^* : P_{i+1} \longrightarrow P_{i}$ such that $f_i^* \circ f_i + f_{i-1} \circ f_{i-1}^* = 0$ for all $2\leq i \leq N- 1$.
Given $1 \leq i \leq j \leq N$ with $j-i$ even we denote by ${}_iP_j$ the following projective $A$-module
$$ {}_iP_j :=  P_i \oplus P_{i+2} \oplus \cdots \oplus P_{j-2} \oplus P_{j}.$$ 
For $1\leq i < j \leq N$ with $j-i$ even we let $d_{i,j} : {}_{i}P_{j} \longrightarrow {}_{i+1}P_{j-1}$ be the morphism of $A$-modules corresponding to the following matrix:
$$ d_{i,j} = \begin{bmatrix} f_{i} & f_{i+1}^* & 0 & \cdots & \cdots & 0 \\
0 & f_{i+2} & f_{i+3}^* & 0 &  & \vdots   \\  \vdots  & \ddots & \ddots& \ddots & \ddots &  \vdots\\ \vdots  & & \ddots & \ddots & \ddots & 0 \\ 0 & \cdots & \cdots & 0 & f_{j-2} & f_{j-1}^* \end{bmatrix}$$
The definition of ${}_iP_j$ extends to any integers $i,j \in \mathbb{Z}$ with the convention that 
\begin{equation}\label{eq:pij}
{}_iP_j = {}_{j+1}P_{i-1}, \quad {}_{i}P_{-j} = {}_{i}P_j, \quad {}_iP_{j\pm2N} = {}_iP_{j}.
\end{equation}
Note that these relations imply ${}_{1-i}P_{j} = {}_{1+i}P_j$ and ${}_{i\pm2N}P_{j} = {}_iP_{j}$.
Furthermore, the definition of $d_{i,j}$ extends naturally to any pair $i,j$ if we set in addition
$$d_{i,i} = (-1)^i f_i^* f_i = (-1)^{i-1} f_{i-1} f_{i-1}^*,$$
a map from ${}_iP_i = P_i$ to  ${}_{i+1}P_{i-1} = P_i$. With this notation one checks that for all $k > 0$ the image of the map
$d_{i\mn k,i\pl k} : {}_{i-k}P_{i+k} \longrightarrow {}_{i-k+1}P_{i+k-1}$ is isomorphic to ${}^{i-k} \sfX^{i+k} \simeq \Omega^k(S_i)$ so that the bounded above complex
$$R_i :=  \cdots \xrightarrow{d_{i\mn k\mn1,i\pl k\pl1}} {}_{i\mn k}P_{i\pl k}  \xrightarrow{d_{i\mn k,i\pl k}} \cdots \xrightarrow{d_{i\mn 2,i\pl 2}}  {}_{i\mn 1}P_{i\pl 1} \xrightarrow{d_{i\mn1,i\pl1}} P_i \longrightarrow 0$$
forms a minimal projective resolution of $S_i$.

\subsection{Generators and relations}
We will consider two kinds of generators for the $\mathsf{Ext}$-algebra, of respective degrees $1$ and $N$. 
We start by defining a map $x_i \in \mathsf{Hom}_{\mathsf{Ho}(A)}(R_i,R_{i+1}[1])$ for any $1\leq i \leq N- 1$. 
Let $k$ be a positive integer. If $k \notin N\mathbb{Z}$, the projective modules ${}_{i\mn k}P_{i\pl k}$ and ${}_{i\pl 1\mn(k\mn1)}P_{i\pl1\pl(k-1)} = {}_{i\mn k\pl2}P_{i\pl k}$ have at least one indecomposable summand in common and we can consider the map $X_{i,k} : {}_{i\mn k}P_{i\pl k} \longrightarrow {}_{i\mn k\pl2}P_{i\pl k}$ given by the identity map on the common factors. If $k \in N+2N\mathbb{Z}$ then from the relations \eqref{eq:pij} we have 
$${}_{i\mn k}P_{i\pl k} = {}_{i\mn N}P_{i\pl N} = {}_{i\pl N\pl1}P_{i\mn N\mn1} ={}_{\mn i\mn N+1}P_{\mn i\pl N\pl1} = P_{N\pl1\mn i }$$
and 
$${}_{i\mn k\pl 2}P_{i\pl k} = {}_{i\mn N\pl2}P_{i\pl N} = {}_{N\mn i}P_{\mn N\mn i} =P_{N\mn i }.$$
In that case we set $X_{i,k} := (-1)^{N-i}f_{N-i}^*$. If $k \in 2N\mathbb{Z}$ then ${}_{i\mn k}P_{i\pl k}  = P_i$, ${}_{i\mn k\pl 2}P_{i\pl k} = {}_{i+2}P_i = P_{i+1}$ and we set $X_{i,k} := (-1)^i f_i$. If $k \geq 0$ we set $X_{i,k} := 0$. Then the family of morphisms of $A$-modules $X_i := (X_{i,k})_{k\in \mathbb{Z}}$ defines a morphism of complexes of $A$-modules from $R_i$ to $R_{i+1}[1]$ and we denote by $x_i$ its image in $\mathsf{Ho}(A)$. 

\smallskip Similarly we define a map $X_i^* : R_{i+1} \longrightarrow R_{i}[1]$ by exchanging the role of $f$ and $f^*$. More precisely we consider in that case $X_{i,-N}^* := (-1)^{N-i}f_{N-i}$ and $X_{i,-2N}^* := (-1)^{i} f_{i}^*$. We denote by $x_i^*$ the image of $X_i^*$ in $\mathsf{Ho}(A)$.

\smallskip
Assume now that $1\leq i \leq N$. The modules ${}_{i\mn k}P_{i\pl k}$ and ${}_{(N\pl1\mn i)\mn (k\mn N)}P_{(N\pl1\mn i)\pl (k\mn N)}$ are equal, which means that starting from the degree $-N$, the terms of the complexes $R_i$ and $R_{N+1-i}[N]$ coincide. We denote by $Y_i : R_i \longrightarrow R_{N+1-i}[N]$ the natural projection between 
$R_i$ and its obvious truncation at degrees $\leq -N$, and by $y_i$ its image in $\mathsf{Ho}(A)$. 

\begin{lemma}\label{lem:relations}
The following relations hold in $\mathrm{End}_{\mathsf{Ho}(A)}^\bullet(\bigoplus R_i)$:
\begin{itemize}
\item[$\mathrm{(a)}$] $ x_1^*\circ x_1 = x_{N-1} \circ x_{N-1}^* = 0$;
\item[$\mathrm{(b)}$] $x_i \circ x_i^* = x_{i+1}^* \circ x_{i+1} $ for all $i = 1,\ldots,N-2$;
\item[$\mathrm{(c)}$] $y_{i+1} \circ x_i  =x_{N-i}^* \circ  y_i  $ for all $i = 1,\ldots,N-1$;
\item[$\mathrm{(d)}$] $y_{i} \circ x_i^* = x_{N-i} \circ y_{i+1}  $ for all $i = 1,\ldots,N-1$.
\end{itemize}

\end{lemma}

\begin{proof}
If $N=1$ there are no relation to check. Therefore we assume $N \geq 2$.
The relations in (a) follow from the fact that $\mathsf{Ext}^2_A(S_1,S_1) = \mathsf{Ext}^2_A(S_N,S_N) =0$, which is for example a consequence of Proposition \ref{prop:poincare} when $N \geq 2$.

\smallskip

To show (c), we observe that the morphism of complexes $X_{i} : R_i \longrightarrow R_{i+1}[1]$ defined above coincide with  $X_{N-i}^*[N] : R_{N+1-i}[N] \longrightarrow R_{N-i}[N+1]$ in degrees less than $-N$. Since $Y_i$ and $Y_{i+1}$ are just obvious truncations we actually have $Y_{i+1} \circ X_i  =X_{N-i}^* \circ  Y_i$. The relation (d) is obtained by a similar argument.

\smallskip

We now consider (b). The morphism of complexes $X_i \circ X_i^*$ and $X_{i+1}^* \circ X_{i+1}$ coincide at every degree $k$ except when $k $ is congruent to $0$ or $-1$ modulo $N$. Let us first look in details at the degrees $-N$ and $-N-1$. The map $X_i \circ X_i^*$ is as follows:
$$ \begin{tikzcd}[ampersand replacement=\&]
\cdots \arrow[r] \arrow[d]\&[10pt] P_{N\mn1\mn i} \oplus P_{N\pl 1\mn i} \arrow[r,"{\begin{bmatrix}  f_{N\mn1\mn i} & f_{N\mn i}^* \end{bmatrix}}"] \arrow[d,"{\begin{bmatrix}  0 & 1 \end{bmatrix}}"] \&[10pt] P_{N\mn i} \arrow[r,"(-1)^{N\mn i}f_{N\mn i}^* \circ f_{N\mn i}"] \arrow[d,"(-1)^{N\mn i}f_{N\mn i}"] \&[10pt] P_{N\mn i} \arrow[d,"{\begin{bmatrix}  1 \\ 0 \end{bmatrix}}"]  \\[40pt]
 P_{N\mn i} \oplus P_{N\pl2\mn i} \arrow[r,"{\begin{bmatrix}  f_{N\mn i} & f_{N\pl1\mn i}^* \end{bmatrix}}"]  \arrow[d,"{\begin{bmatrix}  1 & 0 \end{bmatrix}}"]\& P_{N\pl1\mn i} \arrow[r,"(-1)^{N\mn i}f_{N\mn i} \circ f_{N\mn i}^*"] \arrow[d,"(-1)^{N\mn i}f_{N\mn i}^*"] \& P_{N\pl1\mn i} \arrow[r,"{\begin{bmatrix}  f_{N\mn i}^*\\  f_{N\pl1\mn i} \end{bmatrix}}"] \arrow[d,"{\begin{bmatrix}  0 \\ 1 \end{bmatrix}}"] \& P_{N\mn i} \oplus P_{N\pl2\mn i} \arrow[d] \\[40pt]
P_{N\mn i} \arrow[r,"(-1)^{N\mn i}f_{N\mn i}^* \circ f_{N\mn i}"] \& P_{N\mn i} \arrow[r,"{\begin{bmatrix}  f_{N\mn1 \mn i}^*\\  f_{N\mn i} \end{bmatrix}}"] \& P_{N\mn1 \mn i} \oplus P_{N\pl 1\mn i} \arrow[r] \& \cdots 
\end{tikzcd}
$$
whereas the map $X_{i+1}^* \circ X_{i+1}$ corresponds to the following composition:
$$ \begin{tikzcd}[ampersand replacement=\&]
\cdots \arrow[r] \arrow[d]\&[10pt] P_{N\mn1\mn i} \oplus P_{N\pl 1\mn i} \arrow[r,"{\begin{bmatrix}  f_{N\mn1\mn i} & f_{N\mn i}^* \end{bmatrix}}"] \arrow[d,"{\begin{bmatrix}  1 & 0 \end{bmatrix}}"] \&[10pt] P_{N\mn i} \arrow[r,"(-1)^{N\mn i}f_{N\mn i}^* \circ f_{N\mn i}"] \arrow[d,"(-1)^{N\mn1\mn i}f_{N\mn1\mn i}^*"] \&[10pt] P_{N\mn i} \arrow[d,"{\begin{bmatrix}  0 \\ 1 \end{bmatrix}}"]  \\[40pt]
 P_{N\mn2\mn i} \oplus P_{N\mn i} \arrow[r,"{\begin{bmatrix}  f_{N\mn2\mn i} & f_{N\mn1\mn i}^* \end{bmatrix}}"]  \arrow[d,"{\begin{bmatrix}  0 & 1 \end{bmatrix}}"]\& P_{N\mn1\mn i} \arrow[r,"(-1)^{N\mn1\mn i}f_{N\mn1\mn i}^* \circ f_{N\mn1\mn i}"] \arrow[d,"(-1)^{N\mn1\mn i}f_{N\mn1\mn i}"] \& P_{N\mn1\mn i} \arrow[r,"{\begin{bmatrix}  f_{N\mn2\mn i}^*\\  f_{N\mn1\mn i} \end{bmatrix}}"] \arrow[d,"{\begin{bmatrix}  1 \\ 0 \end{bmatrix}}"] \& P_{N\mn2\mn i} \oplus P_{N\mn i} \arrow[d] \\[40pt]
P_{N\mn i} \arrow[r,"(-1)^{N\mn i}f_{N\mn i}^* \circ f_{N\mn i}"] \& P_{N\mn i} \arrow[r,"{\begin{bmatrix}  f_{N\mn1 \mn i}^*\\  f_{N\mn i} \end{bmatrix}}"] \& P_{N\mn1 \mn i} \oplus P_{N\pl 1\mn i} \arrow[r] \& \cdots 
\end{tikzcd}
$$
We deduce that at the degrees $-N$ and $-N-1$ the map $X_i \circ X_i^* - X_{i+1}^* \circ X_{i+1}$ is given by
$$ \begin{tikzcd}[ampersand replacement=\&]
 P_{N\mn1\mn i} \oplus P_{N\pl 1\mn i} \arrow[r,"{\begin{bmatrix}  f_{N\mn1\mn i} & f_{N\mn i}^* \end{bmatrix}}"] \arrow[d,swap,"{(-1)^{N\mn i} \begin{bmatrix}  f_{N\mn1\mn i} & f_{N\mn i}^* \end{bmatrix}}"] \&[10pt] P_{N\mn i} \arrow[d,"(-1)^{N\mn i}{\begin{bmatrix}  f_{N\mn1\mn i}^* \\ f_{N\mn i} \end{bmatrix}}"] \\[60pt]
P_{N\mn i} \arrow[r,"{\begin{bmatrix}  f_{N\mn1 \mn i}^*\\  f_{N\mn i} \end{bmatrix}}"] \& P_{N\mn1 \mn i} \oplus P_{N\pl 1\mn i} \end{tikzcd} $$
A similar picture holds at the degrees $-2N$ and $-2N-1$:
$$ \begin{tikzcd}[ampersand replacement=\&]
 P_{i} \oplus P_{i\pl2} \arrow[r,"{\begin{bmatrix}  f_{i} & f_{i\pl1}^* \end{bmatrix}}"] \arrow[d,swap,"{(-1)^{i} \begin{bmatrix}  f_{i} & f_{i\pl1}^* \end{bmatrix}}"] \&[10pt] P_{i\pl1} \arrow[d,"(-1)^{i}{\begin{bmatrix}  f_{i}^* \\ f_{i+1} \end{bmatrix}}"] \\[60pt]
P_{i+1} \arrow[r,"{\begin{bmatrix}  f_{i}^*\\  f_{i\pl1} \end{bmatrix}}"] \& P_{i} \oplus P_{i\pl2} \end{tikzcd} $$
Using the map $s : X_{i+1} \rightarrow X_{i+1}[1]$ defined by 
$$s_k := \left\{ \begin{array}{ll} (-1)^{N-i} \mathrm{Id}_{P_{N-i}} & \text{if $-k \in N + 2N\mathbb{N}$}, \\
(-1)^{i} \mathrm{Id}_{P_{i+1}} & \text{if $-k \in 2N + 2N\mathbb{N}$}, \\
 0 & \text{otherwise}, \end{array}\right.$$
 we see that $X_i \circ X_i^* - X_{i+1}^* \circ X_{i+1}$ is null-homotopic, which proves that $x_i \circ x_i^* - x_{i+1}^* \circ x_{i+1}$ is zero in $\mathsf{Hom}_{\mathsf{Ho}(A)}(P_{i+1},P_{i+1}[2])$.
\end{proof}

The next proposition shows that the relations given in Lemma~\ref{lem:relations} are actually enough to describe the $\mathsf{Ext}$-algebra. We use here the concatenation of paths as opposed to the composition of arrows, which explains the discrepancy in the relations.

\begin{prop}\label{prop:main}
The $\mathsf{Ext}$-algebra of $A$ is isomorphic to the path algebra associated with the following quiver
$$ 
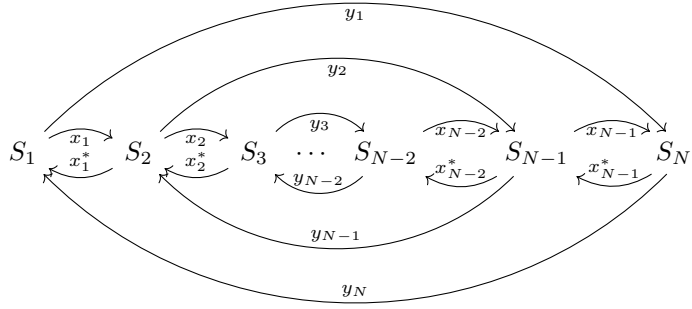
\begin{tikzcd}
S_1 \arrow[r,bend left,swap,"x_1"] \arrow[rrrrrr,bend left=45,swap,"y_1"] & S_2  \arrow[l,bend left,swap,"x_1^*"]  \arrow[r,bend left,swap,"x_2"]  \arrow[rrrr,bend left=45,swap,"y_2"] &  \arrow[l,bend left,swap,"x_2^*"] S_3 \arrow[rr,bend left=45,swap,"y_3"] &[-8mm]  \cdots&[-8mm] S_{N-2} \arrow[ll,bend left=45,swap,"y_{N-2}"]  \arrow[r,bend left,swap,"x_{N-2}"] & S_{N-1} \arrow[l,bend left,swap,"x_{N-2}^*"]  \arrow[llll,bend left=45,swap,"y_{N-1}"]  \arrow[r,bend left,swap,"x_{N-1}"]&  \arrow[l,bend left,swap,"x_{N-1}^*"] S_{N\hphantom{-1}} \arrow[llllll,bend left=45,swap,"y_N"]   
\end{tikzcd} $$
with $x_i$'s of degree $1$ and $y_i$'s of degree $N$, subject to the relations
\begin{itemize}
\item[$\mathrm{(a)}$] $x_1 x_1^* = x_{N-1}^*x_{N-1} = 0$;
\item[$\mathrm{(b)}$] $x_i^* x_i = x_{i+1}x_{i+1}^* $ for all $i = 1,\ldots,N-2$;
\item[$\mathrm{(c)}$] $x_i y_{i+1} = y_i x_{N-i}^* $ for all $i = 1,\ldots,N-1$;
\item[$\mathrm{(d)}$] $x_i^* y_{i} = y_{i+1} x_{N-i} $ for all $i = 1,\ldots,N-1$.
\end{itemize}
\end{prop}

\begin{proof}
Let $Q$ (resp. $I$) be the quiver (resp. the set of relations) given in the proposition. Let $\Gamma = \mathbb{F} Q/\langle I \rangle$ be the corresponding path algebra. By Lemma \ref{lem:relations}, the $\mathsf{Ext}$-algebra of $A$ is a quotient of $\Gamma$. To show that $A \simeq \Gamma$ it is enough to show that the graded dimension of $\Gamma$ is smaller than that of $A$. 

\smallskip

Let $1 \leq i , j \leq N$ and $\gamma$ be a path between $S_i$ and $S_j$ in $Q$ containing only $x_l$'s. Let $k$ be the length of $\gamma$. We have $k \geq |i-j|$, which is the length of the minimal path from $S_i$ to $S_j$. 
Using the relations, there exist loops $\gamma_1$ and $\gamma_2$ around $S_i$ and $S_j$ respectively such that
$$\gamma = \left\{ \begin{array}{ll} \gamma_1 x_i  x_{i+1} \cdots x_{j-1} = x_i  x_{i+1} \cdots x_{j-1} \gamma_2  & \text{if $i\leq j$};\\
 \gamma_1  x_{i-1}^*  x_{i-2}^* \cdots  x_j^* =  x_{i-1}^*  x_{i-2}^* \cdots  x_j^*   \gamma_2 & \text{otherwise}. \end{array}\right.$$
Maximal non-zero loops starting and ending at $S_i$ are either $x_{i-1}^*x_{i-2}^*\cdots x_1^* x_1  x_2 \cdots x_{i-1}$ 
or $x_{i}x_{i+1}\cdots x_{N-1} x_{N-1}^*\cdots   x_{i+1}^* \cdots x_{i}^*$ depending on whether $S_i$ is closer to $S_1$ or $S_N$. Indeed, any longer loop will involve $x_1 x_1^*$ or $x_{N-1}^* x_{N-1}$, which are zero by (a). Therefore if $\mathsf{deg}(\gamma_1) > 2(i-1)$ or 
$\mathsf{deg}(\gamma_1) > 2(N-i)$ then $\gamma_1 =0$. Using a similar argument for loops around $S_j$ we deduce that $\gamma$ is zero whenever 
$$k = \mathsf{deg}(\gamma) > |i-j| + 2\,\mathsf{min}(i-1,j-1, N-j,N-j)$$
which is equivalent to $k= \mathsf{deg}(\gamma) > N-1 - |N+1-j-i|$. This proves that $\gamma$ is zero unless $ |i-j| \leq  k \leq 
N-1 - |N+1-j-i|$ in which case it equals to 
$$\gamma =  x_i  x_{i+1} \cdots x_{r-1} x_{r-1}^* x_{r-2}^*\cdots x_j^* $$
where $k = 2r-i-j$.

\smallskip

Assume now that $\gamma$ is any path between $S_i$ and $S_j$ in $Q$. Using the relations one can write $\gamma$ as $\gamma = y_i^a  \gamma_1  \gamma_2$ where $\gamma_2$ is a loop around $S_j$ containing only $y_l$'s, $\gamma_1$ is a product of $x_l$'s and $a\in\{0,1\}$. Note that $\mathsf{deg}(\gamma_2)$ is a multiple of $2N$ and $\gamma_1$ is either a path from $S_i$ to $S_j$ if $a=0$ or a path from $S_{N+1-i}$ to $S_j$ if $a=1$. From the previous discussion and Proposition \ref{prop:poincare} we conclude that $\gamma$ is zero if $\mathsf{dim}_\mathbb{F}\, \mathsf{Ext}^k_A(S_i,S_j) = 0$ or unique modulo $I$ otherwise. This shows that the projection $\Gamma \twoheadrightarrow A$ must be an isomorphism.
\end{proof}

\end{document}